\documentclass{amsart}

\usepackage{ifxetex,ifluatex}
\if\ifxetex T\else\ifluatex T\else F\fi\fi T%
  \usepackage{fontspec}
\else
  \usepackage[T1]{fontenc}
  \usepackage[utf8]{inputenc}
  \usepackage{lmodern}
\fi

\usepackage{hyperref}
\usepackage[all]{xy}

\newcommand{\CC}{\mathbb{C}}
\newcommand{\FF}{\mathbb{F}}
\newcommand{\QQ}{\mathbb{Q}}
\newcommand{\ZZ}{\mathbb{Z}}
\newcommand{\SageMath}{\textsc{SageMath}}

\newtheorem{theorem}{Theorem}[section]
\newtheorem{lemma}[theorem]{Lemma}
\newtheorem{corollary}[theorem]{Corollary}

\theoremstyle{definition}
\newtheorem{remark}[theorem]{Remark}

\numberwithin{equation}{theorem}

\DeclareMathOperator{\Hom}{Hom}
\DeclareMathOperator{\ord}{ord}
\DeclareMathOperator{\Res}{Res}

\newcommand{\arXiv}[3]{\href{https://arxiv.org/abs/#1}{arXiv:#1v#2} (#3)}

\begin{document}

\title{Geometric decomposition of abelian varieties of order $1$}
\author{Toren D'Nelly-Warady}
\author{Kiran S. Kedlaya}
\thanks{D'Nelly-Warady was supported by the UC LEADS program. Kedlaya was supported by NSF (grants DMS-1802161, DMS-2053473) and UC San Diego (Warschawski Professorship).}

\begin{abstract}
Since the 1970s, the complete classification (up to isogeny) of abelian varieties over finite fields with trivial group of rational points has been known from results of Madan--Pal and Robinson; with two exceptions these are all defined over $\FF_2$. We determine the decomposition of these varieties into simple factors over an algebraic closure of $\FF_2$;
this requires solving a polynomial equation in three roots of unity.
\end{abstract}

\maketitle

\section{Introduction}

We say that an abelian variety $A$ over a finite field $\FF_q$ has \emph{order $1$} if $\#A(\FF_q) = 1$. The simple abelian varieties of order 1 were studied by Madan--Pal \cite{madan-pal}, who showed that there are none for $q \geq 5$, one isogeny class (of genus $1$) for each of $q=3$ and $q=4$, and infinitely many isogeny classes for $q=2$.
More precisely, Madan--Pal produced a family of abelian varieties\footnote{Note that the $A_n$ are only specified by Madan--Pal up to isogeny, not up to isomorphism.} $A_n$ over $\FF_2$ of order 1
with the property that every simple abelian variety over $\FF_2$ of order 1 occurs as an isogeny factor of $A_n$ for exactly one value of $n$.
They also gave a partial analysis of the isogeny factors of the $A_n$, which was completed by Robinson \cite{robinson}. We recall some more details in Theorem~\ref{T:madan-pal}.

The main result of this paper is to determine how the simple abelian varieties over $\FF_2$ of order 1
decompose into isogeny factors after base extension to an algebraic closure $\overline{\FF}_2$ of $\FF_2$. (See \S\ref{sec:geometric decompositions} for the proof.)

\begin{theorem} \label{T:main}
Let $A$ be a simple abelian variety over $\FF_2$ of order $1$.
Let $n$ be the unique positive integer for which $A$ is an isogeny factor of the Madan--Pal variety $A_n$.
Then $A_{\overline{\FF}_2}$ is isogenous to $B^f$ for some simple abelian variety $B$ over $\overline{\FF}_2$, where
\[
f = \begin{cases} 
1 & \mbox{$n$ is a power of $2$ and $n \neq 4$} \\
2 & \mbox{$n$ is not a power of $2$ and $n \neq 7, 30$} \\
2 & n= 4 \\
 3 & n=7 \\
4 & n = 30.
\end{cases}
\]
\end{theorem}

The main step between the Madan--Pal--Robinson classification and Theorem~\ref{T:main} is to identify ways that a ratio of two Frobenius eigenvalues can equal a nontrivial root of unity. This reduces to finding solutions to a certain polynomial equation in three roots of unity; to solve this problem, we apply the paradigm used by Kedlaya--Kolpakov--Poonen--Rubinstein \cite{space-vectors} to classify tetrahedra with rational dihedral angles, incorporating two improvements. First, the classification of minimal additive relations among roots of unity has been extended from weight 12 (by Poonen--Rubinstein \cite[Theorem~3.1]{poonen-rubinstein}) to weight 20 (by Christie--Dykema--Klep \cite[Theorem~4.3]{relations}; see Theorem~\ref{T:relations of weight 16}).
Second, the parallel classification of relations mod 2 has been extended from weight 12 (see \cite[Theorem~6.8]{space-vectors}) to weight 18 (see Theorem~\ref{T:lift minimal relation}).

This analysis leaves a handful of exceptional cases, all involving abelian varieties of dimension at most 6. For these we may appeal directly to the L-Functions and Modular Forms Database (LMFDB) \cite{lmfdb}: \href{https://www.lmfdb.org/Variety/Abelian/Fq/?hst=List&q=2&abvar_point_count=%5B1%5D&simple=yes&search_type=List}{a single query}
returns all simple abelian varieties over $\FF_2$ of order $1$ of dimension at most 6, and we may simply look up the geometric decompositions of these.

As a byproduct, we also determine when two different abelian varieties in the Madan--Pal classification are geometrically isogenous to each other.
\begin{corollary} \label{C:main}
Let $B_1, B_2$ be simple abelian varieties over $\FF_2$ of order $1$.
Then $\Hom(B_{1,\overline{\FF}_q}, B_{2,\overline{\FF}_q}) \neq 0$
if and only if either $B_1$ and $B_2$ are isogenous or there exists a pair $(n_1, n_2)$ with
\[
\{n_1, n_2\} \in \{\{1, 2\}, \{1, 4\}, \{2, 4\}, \{3, 30\}, \{6, 7\}, \{7\}, \{30\}\}
\]
such that $B_i$ is an isogeny factor of $A_{n_i}$ for $i=1,2$.
\end{corollary}

It is natural to ask about other properties of simple abelian varieties of order 1; for instance, how often are they ordinary, or principally polarizable? We discuss some such questions in \S\ref{sec:additional properties} and leave the rest to the interested reader.

Some of our arguments depend on computer calculations made in \SageMath \cite{sage}. We have collected these in a series of Jupyter notebooks and made them available via a GitHub repository \cite{repo}.
We also use the code accompanying \cite{space-vectors} for computing torsion closures of ideals in Laurent polynomial ideals; see \cite{space-vectors-repo}.

\section{Weil polynomials}

Throughout this section and the next, let $\FF_q$ be a finite field of characteristic $p$. (We will eventually take $p = q =2$, but we do not impose this restriction yet.)

By a \emph{Weil polynomial} (or more precisely a \emph{$q$-Weil polynomial} if we need to specify $q$), we will mean a monic integer polynomial $Q(x)$ of some even degree $2g$ such that every root of $Q(x)$ in $\CC$ has absolute value $q^{1/2}$.
Note that this implies that
\begin{equation} \label{eq:functional equation}
Q(x) = \pm  q^{-g} x^{2g} Q(q/x).
\end{equation}
By a \emph{real Weil polynomial}, we will mean a monic integer polynomial $R(x)$ with roots in the interval $[-2\sqrt{q}, 2\sqrt{q}]$. For every real Weil polynomial, the formula
\begin{equation} \label{eq:Weil poly to real Weil poly}
Q(x) = x^{\deg R(x)} R(x + qx^{-1})
\end{equation}
defines a Weil polynomial $Q(x)$; the Weil polynomials that occur in this way are precisely the ones for which the plus sign occurs in \eqref{eq:functional equation}.

Now let $A$ be an abelian variety over $\FF_q$. By results of Weil, there is a Weil polynomial $Q(x)$ which equals the characteristic polynomial of Frobenius on the $\ell$-adic Tate module for every prime $\ell$ not dividing $q$.
Moreover, the plus sign occurs in \eqref{eq:functional equation}, so we may associate to $A$ both a Weil polynomial and a real Weil polynomial.

We say that $A$ is \emph{simple} if it is nonzero and not isogenous to the product of two nonzero abelian varieties.
Since multiplication of Weil polynomials corresponds to taking products of abelian varieties, any abelian variety with irreducible Weil polynomial is simple. The converse is not quite true because of the Honda--Tate theorem; see \S\ref{sec:Honda-Tate}.

By the \emph{Newton polygon} of $A$, we will mean the Newton polygon of its associated Weil polynomial $Q(x)$ with respect to the $p$-adic valuation $v_q$ for the normalization $v_q(q) = 1$.
It can also be computed from the Newton polygon of the real Weil polynomial $R(x)$ (for the same normalized valuation): by \eqref{eq:Weil poly to real Weil poly}, all slopes in $[0, 1/2)$ have the same multiplicity in both polygons.
(Explicitly, if $v_q(\alpha) \in [0,1/2)$, then $v_q(q\alpha^{-1}) \in (1/2, 1]$ and so $v_q(\alpha + q\alpha^{-1}) = v_q(\alpha)$.)
For each $s \in [0,1)$, the multiplicity $m_s$ of $s$ as a slope in the Newton polygon has the property that $m_s s \in \ZZ$, e.g., from Manin's description of the formal group of $A$ in terms of the quantities $m_s s$ \cite[Theorem~4.1]{manin};
in other words, the Newton polygon of $A$ has integral vertices.
We say that $A$ is \emph{ordinary} if the slopes of its Newton polygon are all equal to 0 or 1.

For $n$ a positive integer, let $A_{\FF_{q^n}}$ denote the base extension of $A$ from $\FF_q$ to $\FF_{q^n}$. From the definition of the Weil polynomial, we see that the Weil polynomial $Q_n(x)$ of $A_{\FF_{q^n}}$ has roots which are the $n$-th powers of the roots of $Q(x)$; that is,
\[
Q_n(x) = \Res_y(Q(y), y^n - x)
\]
where $\Res$ denotes the resultant.

We say that $A$ is \emph{geometrically simple} if $A_{\FF_{q^n}}$ is simple for each positive integer $n$, or equivalently if $A_{\overline{\FF}_q}$ is simple.
\begin{lemma} \label{lem:geometrically simple criterion}
Let $A$ be a simple abelian variety over $\FF_q$ with irreducible Weil polynomial $Q(x)$. If no two roots of $Q(x)$ have ratio equal to a nontrivial root of unity, then $A$ is geometrically simple.
\end{lemma}
\begin{proof}
Let $\pi \in \overline{\QQ}$ be a root of $Q(x)$.
Let $f$ be the integer $[\QQ(\pi):\QQ(\pi^n)]$.
If $\pi_1,\dots,\pi_d$ are the conjugates of $\pi$ listed without repetition, then $d = [\QQ(\pi):\QQ]$; $\pi_1^n, \dots, \pi_d^n$ are the conjugates of $\pi^n$; and the number of distinct entries in this list is $[\QQ(\pi^n):\QQ]$. It follows that $Q_n(x)$ is the $f$-th power of an irreducible polynomial.

In the above notation, the condition on $Q(x)$ implies that $f=1$ for all $n$, so $Q_n(x)$ is irreducible for all $n$.
\end{proof}

\begin{lemma} \label{lem:geometrically simple by Newton polygon}
Let $A$ be an abelian variety over $\FF_q$ of dimension $g > 2$ with Weil polynomial $Q(x)$. Suppose that the Newton polygon of $A$ consists of the slopes $1/g$ and $1-1/g$, each with multiplicity $g$. Then $A$ is geometrically simple.
\end{lemma}
\begin{proof}
Since the condition on the Newton polygon is stable under base extension, it suffices to check that $A$ is simple. For this, note that there is no way to separate the slopes of the Newton polygon into two sets of slopes which each satisfy the conditions on the Newton polygon of an abelian variety
(namely, integer vertices and symmetry of slopes under $s \mapsto 1-s$).
\end{proof}

\section{The Honda--Tate theorem}
\label{sec:Honda-Tate}

We summarize \cite[Theorem 8, Theorem 9]{waterhouse-milne}.

\begin{theorem}[Honda--Tate] \label{T:Honda-Tate}
There is a one-to-one correspondence between isogeny classes of simple abelian varieties over $\FF_q$ and irreducible $q$-Weil polynomials,
in which a simple abelian variety $A$ corresponds to an irreducible $q$-Weil polynomial $Q(x)$ with the property that the Weil polynomial of $A$ equals $Q(x)^e$ for some positive integer $e$.
More precisely, $e$ is the least common denominator of the following rational numbers indexed by places $v$ of the number field $\QQ(\pi) = \QQ[x]/(Q(x))$:
\begin{itemize}
\item
If $v$ is real: $\frac{1}{2}$.
\item
If $v$ is complex: $0$.
\item
If $v$ lies over a finite prime not equal to $p$: $0$.
\item
If $v$ lies over $p$: $\frac{\ord_v(\pi)}{\ord_v(q)} [\QQ(\pi)_v:\QQ_p]$.
\end{itemize}
\end{theorem}
More precisely, the endomorphism algebra of $A$ is a division algebra with center $\QQ(\pi)$ and the listed quantities are the Brauer invariants of this division algebra.

\begin{corollary}  \label{cor:Honda-Tate p}
For every simple abelian variety $A$ over $\FF_p$, the Weil polynomial of $A$ is irreducible \emph{unless} it equals $(x^2-p)^2$.
\end{corollary}
\begin{proof}
In Theorem~\ref{T:Honda-Tate}, the invariant associated to every place above $p$ is an integer (namely the degree of the residue field of $\QQ(\pi)_v$ over $\FF_p$). Hence the only way to get a nontrivial denominator is for $\QQ(\pi)$ to have a real place, which occurs if and only if 
$Q(x) = x^2 - p$.
\end{proof}

\begin{corollary} \label{cor:Honda-Tate ordinary}
For every \emph{ordinary} simple abelian variety $A$ over $\FF_q$, the Weil polynomial of $A$ is irreducible.
\end{corollary}
\begin{proof}
In Theorem~\ref{T:Honda-Tate}, the invariant associated to every place above $p$ is an integer because $\ord_v(\pi)/\ord_v(q) \in \{0,1\}$. Hence the only way to get a nontrivial denominator is for $\QQ(\pi)$ to have a real place, which occurs if and only if 
$Q(x) = x^2 - q$ or $Q(x) = x \pm \sqrt{q}$; but these cases are not ordinary.
\end{proof}

\section{The Madan--Pal--Robinson classification}
\label{sec:classification}

We describe the classification of abelian varieties over $\FF_2$ of order 1 in terms of their associated real Weil polynomials. This is due to Madan--Pal \cite{madan-pal} modulo the computation of certain irreducible factors, which was completed by Robinson \cite{robinson}.

For $n$ a positive integer, define the integer polynomial
\begin{equation} \label{eq:madan-pal polynomial}
P_n(x) = \prod_{0 \leq k \leq n/2, \gcd(n,k) = 1} \left( x^2 - (4 + 2 \cos \frac{2\pi k}{n})x + 1\right)
\end{equation}
of degree $\max\{2, \phi(n)\}$.
For $n = 2, 7, 30$, $P_n(x)$ is reducible:
\begin{align*}
P_2(x) &= (x-1)^2 \\
P_7(x) &= (x^3 - 5x^2 + 6x - 1)(x^3 - 6x^2 + 5x - 1) \\
P_{30}(x) &= (x^4 - 8x^3 + 14x^2 - 7x + 1)(x^4 - 7x^3 + 14x^2 - 8x + 1).
\end{align*}

\begin{theorem}[Madan--Pal, Robinson] \label{T:madan-pal}
The real Weil polynomials associated to simple abelian varieties over $\FF_2$ of order $1$ are precisely those of the form
$P(3-x)$ where $P(x)$ equals either $P_n(x)$ for some positive integer $n \neq 2,7,30$ or an irreducible factor of one of $P_2(x), P_7(x), P_{30}(x)$.
\end{theorem}
\begin{proof} 
We first check that Corollary~\ref{cor:Honda-Tate p} applies to show that $P_n(3-x)$ is the real Weil polynomial of some abelian variety $A_n$ over $\FF_2$ of order $1$. 
For this it suffices to verify that $x^2-2$ never occurs as a factor of $P_n(3-x)$, 
or equivalently that $(3-x)^2 - 2 = x^2 - 6x + 7$ never occurs as a factor of $P_n(x)$; this is apparent because $P_n(0) = 1$. 

By \cite[Theorem~4]{madan-pal}, every simple abelian variety over $\FF_2$ of order 1 is a factor of $A_n$ for a unique value of $n$.
It thus remains to compute the factorizations of the polynomials $P_n(x)$; this is established in \cite{robinson} based on partial results from \cite{madan-pal}.
\end{proof}

We next compute the Newton polygons of these abelian varieties.

\begin{lemma} \label{lem:Pn value at 1}
For $n = 2^m$ with $m \geq 2$, $P_n(1) = (-1)^{n/4} 2$.
\end{lemma}
\begin{proof}
Let $\zeta$ be a primitive $n$-th root of unity. We may then write
\begin{align*}
P_n(1) &= \prod_{0 \leq k \leq n/2, \gcd(n,k) = 1} \left( -2 - 2\cos \frac{2\pi k}{n} \right) \\
&= 
(-1)^{n/4} \prod_{0 \leq k \leq n/2, \gcd(n,k) = 1} (2 + \zeta^{k} + \zeta^{-k})\\
&= (-1)^{n/4} \prod_{0 \leq k \leq n/2, \gcd(n,k) = 1} (1 +\zeta^k)(1+\zeta^{-k}) \\
&= (-1)^{n/4} \Phi_n(-1)
\end{align*}
where $\Phi_n(x) = x^{n/2} + 1$ is the $n$-th cyclotomic polynomial. Since evidently $\Phi_n(-1) = 2$, this proves the claim.
\end{proof}

\begin{lemma} \label{lem:Newton polygon}
Let $A_n$ be the abelian variety with real Weil polynomial $P_n(3-x)$.
\begin{enumerate}
\item[(a)] If $n$ is not a power of $2$, then $A_n$ is ordinary.
\item[(b)] If $n$ is a power of $2$, then the Newton polygon of $A_n$ has all slopes equal to $1/2^m$ or $1-1/2^m$ where $m = \max\{1, \log_2(n)-1\}$.
\end{enumerate}
\end{lemma}
\begin{proof}
Write $n = 2^m j$ with $j$ odd. 
Over the ring of algebraic integers, \eqref{eq:madan-pal polynomial} implies
\begin{equation} \label{eq:Pn mod 2 congruence}
P_n(x) \equiv \prod_{0 \leq k \leq n/2, \gcd(n,k) = 1} \left( x^2 - (2 \cos \frac{2\pi k}{n})x + 1\right) \equiv \Phi_n(x)
\pmod{2};
\end{equation}
in particular, each root of $P_n(x)$ is congruent modulo a prime above 2 to a primitive $j$-th root of unity.
In case (a), this implies that none of the roots of $P_n(3-x)$ have positive valuation,
as each root is congruent modulo a prime above 2 to $1 - \eta$ where $\eta$ is a primitive $j$-th root of unity for some $j>1$; consequently, $A_n$ is ordinary. In case (b), it implies that all of the roots of $P_n(3-x)$ have positive valuation;
this automatically implies the claim for $m=0,1$.
For $m \geq 2$, we also need to know that $P_n(3) \equiv 2 \pmod{4}$
in order to deduce that $P_n(3-x)$ is an Eisenstein polynomial at 2;
this follows by applying Lemma~\ref{lem:Pn value at 1} and \eqref{eq:Pn mod 2 congruence}
to obtain
\[
P_n(3) \equiv P_n(1) + 2 P'_n(1) \equiv P_n(1) + 2 \Phi'_n(1) \equiv 2 \pmod{4}. \qedhere
\]
\end{proof}

\section{Reduction to an equation in roots of unity}

Let $A_n$ be an abelian variety over $\FF_2$ with real Weil polynomial $P_n(3-x)$. In light of Lemma~\ref{lem:geometrically simple criterion},
the key step in the proof of Theorem~\ref{T:main} will be to determine when the ratio of two Frobenius eigenvalues of $A_n$ can equal a nontrivial root of unity. The following lemmas reduce this to a tractable problem, which we solve in the remainder of the paper.

\begin{lemma} \label{lem:MP eigenvalue}
For any positive integer $n$, the Frobenius eigenvalues of $A_n$
are precisely the numbers $\alpha$ satisfying an equation of the form
\begin{equation} \label{eq:MP eigenvalue}
\alpha - 2\eta \alpha^{-1} + \eta - 1 = 0
\end{equation}
for some primitive $n$-th root of unity $\eta$.
\end{lemma}
\begin{proof}
By definition, $\alpha$ is a Frobenius eigenvalue of $A_n$ if and only if 
$\alpha + 2\alpha^{-1}$ is a root of the real Weil polynomial $P_n(3-x)$. From \eqref{eq:madan-pal polynomial}, the latter condition means that there exists a primitive $n$-th root of unity $\eta$ such that
\[
(3 - \alpha - 2\alpha^{-1}) + (3 - \alpha - 2\alpha^{-1})^{-1}
= 4 + \eta + \eta^{-1}.
\]
This equation simplifies to
\[
(\alpha - 2\eta \alpha^{-1} + \eta - 1)
(\alpha - 2\eta^{-1} \alpha^{-1} + \eta^{-1} - 1) = 0,
\]
and so holds if and only if 
one of the two factors on the left vanishes.
If it is the first factor, then \eqref{eq:MP eigenvalue} holds;
otherwise, \eqref{eq:MP eigenvalue} becomes true after we replace $\eta$ with $\eta^{-1}$.
\end{proof}

\begin{lemma} \label{lem:roots of unity ratio}
Let $\alpha_1, \alpha_2$ be distinct Frobenius eigenvalues of $A_n$.
If $\eta_3 = \alpha_1/\alpha_2$ is a root of unity (of any order), then there exist roots of unity $\eta_1, \eta_2$ of order $n$ such that
$g(\eta_1, \eta_2, \eta_3) = 0$, where
\begin{align} \label{eq:equation for g}
g(z_1, z_2, z_3) &= z_1 + z_1^{-1} + z_2 + z_2^{-1} \\
\nonumber &\quad + z_3 + z_3^{-1} - z_1 z_3^{-1} - z_1^{-1} z_3 - z_2 z_3^{-1} - z_2^{-1} z_3 + z_1 z_2 z_3^{-1} + z_1^{-1} z_2^{-1} z_3\\ 
 \nonumber
&\quad  - 2 z_1 z_2 z_3^{-2} - 2 z_1^{-1} z_2^{-1} z_3^2
\end{align}
as an element of the Laurent polynomial ring $R = \ZZ[z_1^{\pm},z_2^{\pm}, z_3^{\pm}]$.
\end{lemma}
\begin{proof}
By Lemma~\ref{lem:MP eigenvalue}, there exist roots of unity $\eta_1, \eta_2$ of order $n$ such that
\[
\alpha_1 - 2 \eta_1 \alpha_1^{-1} + \eta_1 - 1 = \alpha_2 - 2 \eta_2^{-1} \alpha_2^{-1} + \eta_2^{-1} - 1 = 0.
\]
We obtain \eqref{eq:equation for g} by substituting $\alpha_1 = \alpha_2 \eta_3$ in the first equation to obtain
\begin{align*}
0 &= \alpha_2^2\eta_3^2 + \alpha_2 \eta_3(\eta_1-1) - 2 \eta_1\\
0 &=  \alpha_2^2 + \alpha_2 (\eta_2^{-1} - 1) - 2 \eta_2^{-1},
\end{align*}
then eliminating $\alpha_2$. This is an easy resultant computation in \SageMath{} but can also be done by hand: subtract $\eta_3^2$ times the second equation from the first to obtain
\[
(\eta_3 + \eta_2^{-1} \eta_3^2 - \eta_3^2 - \eta_1 \eta_3) \alpha_2 = 2 \eta_2^{-1} \eta_3^2 - 2\eta_1,
\]
then substitute for $\alpha_2$.
\end{proof}

\section{Additive relations among roots of unity}
\label{sec:relations}

Our approach to solving the equation $g(\eta_1,\eta_2,\eta_3) = 0$ from Lemma~\ref{lem:roots of unity ratio}
is to treat it as an additive relation among 16 roots of unity.
A strategy for analyzing such relations was described by Conway--Jones \cite{conway-jones} and used thereafter by numerous authors; a recent example is \cite{space-vectors}, whose notation and terminology we follow (with some improvements as noted in the introduction).

Let $\mu$ be the multiplicative group of roots of unity in $\CC$ and let $\ZZ[\mu] \subset \CC$ be the ring of cyclotomic integers. For $N$ a positive integer, let $\mu_N$ be the subgroup of $\mu$ generated by $\zeta_N = e^{2 \pi i/N}$.
By a \emph{cyclotomic relation}, we will mean a multisubset (subset with multiplicity) of $\mu$ with sum zero; this corresponds to the notion of a \emph{sorou} (acronym for \emph{sum of roots of unity}) in \cite{relations}.
By a \emph{mod $2$ cyclotomic relation}, we will mean a multisubset of $\mu$ with sum divisible by 2 in $\ZZ[\mu]$. 

A cyclotomic relation (resp. a mod 2 cyclotomic relation) is \emph{indecomposable} if it is nonempty and cannot be partitioned into two nonempty cyclotomic relations (resp. mod 2 cyclotomic relations).
The cardinality of a cyclotomic relation (resp. a mod 2 cyclotomic relation) is also called its \emph{weight}.

\begin{theorem} \label{T:relations of weight 16}
Every indecomposable cyclotomic relation of weight at most $20$ is a subset of $\mu$; that is, its elements are pairwise distinct.
(The bound is best possible; see Remark~\ref{rem:weight 21 example}.)
\end{theorem}
\begin{proof}
This follows from the explicit classification of minimal cyclotomic relations of weight at most 20 \cite[Theorem~4.3]{relations}. For the case of weight at most $12$, see also \cite[Theorem~3.1]{poonen-rubinstein}.
\end{proof}

For smaller weight, we will also need the explicit forms of such sequences.
\begin{theorem} \label{T:relations of weight 8}
Let $S$ be an indecomposable cyclotomic relation of weight at most $8$.
Then then there exists $\zeta \in \mu$ such that $\zeta^{-1} S = \{\zeta^{-1} \eta: \eta \in S\}$ appears in
Table~\ref{table:minimal relations}.
\end{theorem}
\begin{table}[ht]
\begin{tabular}{c|c|c}
$n$ & Type & Relation \\ \hline
2 & $R_2$ & $1, -1$ \\ \hline
3 & $R_3$ & $1, \zeta_3, \zeta_3^2$ \\ \hline
5 & $R_5$ & $1, \zeta_5, \zeta_5^2, \zeta_5^3, \zeta_5^4$ \\ \hline
6 & $(R_5:R_3)$ & $\zeta_5, \zeta_5^2, \zeta_5^3, \zeta_5^4, -\zeta_3, -\zeta_3^2$ \\ \hline
7 & $R_7$ & $1, \zeta_7, \zeta_7^2, \zeta_7^3, \zeta_7^4,\zeta_7^5,\zeta_7^6$ \\
7 & $(R_5:2R_3)$ & $1,\zeta_5^2, \zeta_5^3, -\zeta_3\zeta_5, -\zeta_3^2\zeta_5, -\zeta_3 \zeta_5^4, -\zeta_3^2 \zeta_5^4$ \\
7 & $(R_5:2R_3)$ & $1,\zeta_5, \zeta_5^4, -\zeta_3 \zeta_5^2, -\zeta_3^2 \zeta_5^2, -\zeta_3 \zeta_5^3, -\zeta_3^2 \zeta_5^3$\\ \hline
8 & $(R_5:3R_3)$ & $\zeta_5^2, \zeta_5^3, -\zeta_3, -\zeta_3^2, -\zeta_3 \zeta_5, -\zeta_3^2 \zeta_5, -\zeta_3 \zeta_5^4, -\zeta_3^2 \zeta_5^4$\\
8 & $(R_5:3R_3)$ & $\zeta_5, \zeta_5^4, -\zeta_3, -\zeta_3^2, -\zeta_3 \zeta_5^2, -\zeta_3^2 \zeta_5^2, -\zeta_3 \zeta_5^3, -\zeta_3^2 \zeta_5^3$\\
8 & $(R_7:R_3)$ & $\zeta_7, \zeta_7^2, \zeta_7^3, \zeta_7^4, \zeta_7^5, \zeta_7^6, -\zeta_3, -\zeta_3^2$
\end{tabular}
\caption{Indecomposable cyclotomic relations of weight at most $8$.}
\label{table:minimal relations}
\end{table}
\begin{proof}
The result is originally due to W\l odarski \cite{wlodarski} and was extended to weight $9$ by Conway--Jones \cite{conway-jones}.
Our notation follows \cite[Table~3.1]{poonen-rubinstein}.
\end{proof}

In many cases, one is particularly interested in cyclotomic relations which are stable under complex conjugation.
Any such relation can be partitioned into indecomposable relations, but not uniquely; hence it is not automatic that there is such a partition which is stable under complex conjugation.
We first address this question in the easier mod 2 setting, recalling \cite[Lemma~6.9]{space-vectors}; see Lemma~\ref{lem:conjugate stable relations} for an analogous statement for genuine cyclotomic relations.

\begin{lemma} \label{lem:conjugate stable relations mod 2}
Let $S$ be a mod $2$ cyclotomic relation (of any weight) which is stable under complex conjugation. Then $S$ can be partitioned into indecomposable mod $2$ cyclotomic relations in a manner which is itself stable under complex conjugation;  that is, the conjugate of each part is also a part.
\end{lemma}
\begin{proof}
We induct on the weight of $S$.
Let $T$ be an indecomposable mod 2 cyclotomic relation contained in $S$. Let $\overline{T}$ be the complex conjugate of $T$. If either $T = \overline{T}$ or $T \cap \overline{T} = \emptyset$, we apply the induction hypothesis to the complement of $T \cup \overline{T}$ in $S$; otherwise, we apply it to the symmetric difference $T \oplus \overline{T}$ and to its complement in $S$.
\end{proof}

The following is an extension of \cite[Theorem~6.7, Corollary~6.8]{space-vectors}. The bound is best possible; see again Remark~\ref{rem:weight 21 example}.
\begin{theorem} \label{T:lift minimal relation}
Let $S = \{\eta_1, \dots, \eta_m\}$ be a mod $2$ cyclotomic relation of weight at most $18$. Then there exist signs $\sigma_1,\dots,\sigma_m \in \{\pm 1\}$ such that $\sigma_1 \eta_1 + \cdots + \sigma_m \eta_m = 0$; that is, $S$ can be ``lifted'' to a cyclotomic relation.
\end{theorem}
\begin{proof}
We proceed by induction on $m$, using \cite[Theorem~6.7]{space-vectors} to treat the cases $m \leq 12$ as base cases. For the induction step, we may assume that $S$ is indecomposable and $m \geq 13$.
Let $N$ be the \emph{level} of $S$ in the sense of \cite[\S 6]{space-vectors},
i.e., the smallest positive integer for which $S \subseteq \mu_N$.
We may assume that $S$ is \emph{minimal} in the sense that no rotation of $S$ (i.e., the product of $S$ with a root of unity) has level strictly less than $N$;
in this case $N$ is odd and squarefree \cite[Lemma~6.1]{space-vectors}.
Let $p$ be the largest prime factor of $N$; since $m > 10$ we must have $p \geq 7$. We can write $S$ as a disjoint union $\bigsqcup_{i=0}^{p-1} \zeta_p^i T_i$ where $T_i$ is a nonempty set of roots of unity whose orders are not divisible by $p$; the sums of the $T_i$ are pairwise congruent modulo 2 (see the proof of \cite[Lemma~6.4]{space-vectors}).

If some $T_i$ is empty, then each $T_i$ is itself a mod 2 cyclotomic relation; this gives a contradiction against either the hypothesis that $S$ is indecomposable (if there is more than one nonempty $T_i$) or the definition of the level (if only one $T_i$ is nonempty).

If some $T_i$ is a singleton, we may assume without loss of generality that $T_0 = \{1\}$. Then for each $i>0$, $T_i \cup \{1\}$ is itself a mod 2 cyclotomic relation of weight at most $n-p+1 < n$,
so by the induction hypothesis we can find a function $f_i: T_i \to \{\pm 1\}$ such that
\[
1 + \sum_{\eta \in T_i} f_i(\eta) \eta = 0.
\]
Combining these, we obtain a lift of $S$ to a cyclotomic relation.

We may thus assume that $\#T_i \geq 2$ for each $i$.
This rules out $p \geq 11$, as this would imply $m \geq \sum_i \#T_i \geq 22$;
we thus have $p = 7$. We also must have $\#T_i \leq 6$ for each $i$,
as otherwise $m \geq \sum_i \#T_i \geq 7 + 6 \times 2 = 21$.
We may assume without loss of generality that $T_0 = \{1, \zeta_N\}$ where
$N \in \{3,5,15\}$.
Then each $T_i$ consists of powers of $\zeta_{15}$, has sum congruent to $1+\zeta_N$ mod 2, and contains no mod 2 cyclotomic relation; it is straightforward to enumerate all sets satisfying these conditions (of cardinality at most 6) using \SageMath.

It now suffices to show that there exist functions $f_i: T_i \to \{\pm 1\}$ for $i=1,\dots,6$ such that the sums $\sum_{\eta \in T_i} f_i(\eta) \eta$ are either all equal to $1 + \zeta_N$ or all equal to $1-\zeta_N$. 
From the enumeration, we see that for $N \in \{3,5\}$, we can always 
achieve $\sum_{\eta \in T_i} f_i(\eta) \eta = 1 + \zeta_N$.
When $N = 15$, there do exist indices $i \neq j \in \{1,\dots,6\}$ such 
that for any functions $f_i: T_i \to \{\pm 1\}$, $f_j: T_j \to \{\pm 1\}$,
\[
\sum_{\eta \in T_i} f_i(\eta)\eta \neq 1 + \zeta_N, \qquad 
\sum_{\eta \in T_j} f_j(\eta)\eta \neq 1 - \zeta_N;
\]
however, for all such indices we have
\[
\#T_i \geq 5, \qquad \#T_j \geq 4
\]
and this cannot occur when $n \leq 18$.
\end{proof}

\begin{remark} \label{rem:weight 21 example}
In \cite[\S 2.2]{relations}, an example is given of an indecomposable cyclotomic relation of weight 21 with one pair of repeated elements. Omitting this pair gives a mod 2 cyclotomic relation of weight 19 which cannot be lifted to a cyclotomic relation; one can extend the proof of Theorem~\ref{T:lift minimal relation} to give a complete classification of such relations.
\end{remark}

\begin{remark} \label{rem:unique lift}
In Theorem~\ref{T:lift minimal relation}, if $S$ is indecomposable (mod 2), then the choice of the $\sigma_i$ is unique up to multiplying them all by $-1$,
by the following reasoning. Suppose that $\sigma'_1,\dots, \sigma'_m \in \{\pm 1\}$ is a second choice
for which $\sum_i \sigma'_i \eta_i = 0$.
Define the partition $\{1,\dots,m\} = J_1 \sqcup J_2$ by
\[
J_1 = \{i \in \{1,\dots,m\}: \sigma_i = \sigma'_i\}, \qquad
J_2 = \{i \in \{1,\dots,m\}: \sigma_i \neq \sigma'_i\};
\]

Then both $\{\sigma_i \eta_i: i \in J_1\}$ and $\{\sigma_i \eta_i: i \in J_2\}$ are cyclotomic relations, and hence mod 2 cyclotomic relations. Since $S$ is indecomposable as a mod 2 cyclotomic relation, this is only possible if $J_1 = \emptyset$ or $J_2 = \emptyset$.
\end{remark}

The following is an extension of \cite[Lemma~4.1]{poonen-rubinstein}, but with a different proof. See also Remark~\ref{rem:conjugate stable relations proof}.

\begin{lemma} \label{lem:conjugate stable relations}
Let $S$ be a cyclotomic relation of weight at most $18$ which is stable under complex conjugation. Then $S$ can be partitioned into indecomposable cyclotomic relations in a manner which is itself stable under complex conjugation; that is, the conjugate of each part is also a part.
\end{lemma}
\begin{proof}
We induct on the weight of $S$.
By Lemma~\ref{lem:conjugate stable relations mod 2} $S$ admits a conjugation-stable partition as a mod 2 cyclotomic relation. 
Apply Theorem~\ref{T:lift minimal relation} to lift each part to a genuine cyclotomic relation; by Remark~\ref{rem:unique lift} each of these lifts is unique up to an overall sign, so we can ensure that conjugate parts lift to conjugate relations. 

This yields a conjugation-equivariant function $f: S \to \{\pm 1\}$ such that
$\sum_{\eta \in S} f(\eta) \eta = 0$. If $f$ is constant, we have a partition of $S$ of the desired form; otherwise as in Remark~\ref{rem:unique lift} the level sets of $f$ partition $S$ nontrivially into two cyclotomic relations, each stable under complex conjugation.
\end{proof}

\begin{remark} \label{rem:conjugate stable relations proof}
In this paper, we will only need Lemma~\ref{lem:conjugate stable relations}
for weight at most 16. We describe an alternate proof of this restricted result in the style of the original proof of \cite[Lemma~4.1]{poonen-rubinstein}.

Let $T$ be an indecomposable cyclotomic relation of minimal weight contained in $S$. Let $\overline{T}$ be the complex conjugate of $T$; it will suffice to check that either $T = \overline{T}$ or $T \cap \overline{T} = \emptyset$,
as then we can remove $T \cup \overline{T}$ from $S$ and then apply the induction hypothesis to conclude.

If $T$ is of type $R_p$ for some prime $p$ in the sense of Table~\ref{table:minimal relations} (i.e., a collection of $p$ equally spaced roots of unity), then $\overline{T}$ is a rotation of $T$ and so the claim is evident. In particular, this covers all cases where $T$ has weight at most 5; since $S$ has weight at most $17$, this also covers all cases where $S$ can be partitioned into three or more indecomposable cyclotomic relations.

The remaining case is where the complement $T'$ of $T$ in $S$ is itself an indecomposable cyclotomic relation and both $T$ and $T'$ have weights in the range $\{6,7,8\}$. By Theorem~\ref{T:relations of weight 8}, 
there exists $\zeta \in \mu$ such that $\zeta^{-1} T$ is listed in Table~\ref{table:minimal relations}. Since $T \cap \overline{T} \neq \emptyset$, there must be two (not necessarily distinct) entries $\eta_i, \eta_j$ in the listed sequence such that $\zeta \eta_i = (\zeta \eta_j)^{-1}$, yielding
\[
\zeta^2 = \eta_i^{-1} \eta_j^{-1}.
\]
This limits $T$ to a computable finite list of options, and similarly for $T'$.

Write $T$ as the disjoint union of $T_1$ and $T_2$ where $T_1 = T \cap \overline{T}$, and similarly for $T'$. Then each of $T_2$ and $T'_2$ is disjoint from its complex conjugate, and yet together they are stable under complex conjugation; they must therefore be conjugates of each other. In particular, 
\[
\sum T_1 = -\sum T_2 = \overline{-\sum T'_2} = \overline{\sum T'_1} = \sum T'_1.
\]
We may thus check the claim by enumerating the candidates for $T$; for each $T$, computing the partition $T_1 \sqcup T_2$ and the sum $t = \sum T_1$; finding pairs $T \neq T'$ with matching values of $t$; and checking that for each such pair, the union $T \cup T'$ contains a cyclotomic relation which is either stable under complex conjugation or disjoint from its complex conjugate.
This computation takes a few minutes in \SageMath.
\end{remark}

\begin{remark} \label{remark:mod 2}
For $\eta_1, \eta_2 \in \mu$, if $\eta_1 - \eta_2$ has 2-adic valuation greater than 1 with respect to some prime above 2 in $\ZZ[\mu]$,
then $\eta_1 = \eta_2$. To see this, suppose by way of contradiction that $\eta_1/\eta_2$ has order $m>1$.  If $m = \ell^e$ for some prime $\ell$ and some positive integer $e$,
then by \cite[Theorem~10.1]{janusz} the 2-adic valuation of $\eta_1 - \eta_2$ is 0 if $\ell\neq 2$ or $1/\phi(m) \leq 1$ if $\ell = 2$. Otherwise, $\eta_1 - \eta_2$ is a unit in $\ZZ[\mu]$
\cite[\S I.10, Exercise 2]{janusz}.

If instead $\eta_1-\eta_2$ has 2-adic valuation greater than $1/2^n$ for some positive integer $n$, then $\eta_1^2 - \eta_2^2 = (\eta_1 - \eta_2)(\eta_1 - \eta_2 + 2\eta_2)$ has 2-adic valuation at least $1/2^{n-1}$.
By induction, we deduce that $\eta_1^{2^n} = \eta_2^{2^n}$.
\end{remark}

\section{Solving an equation in roots of unity}

We now apply the results from \S\ref{sec:relations}, plus the computation of torsion closures described in \cite[\S 7]{space-vectors}, to solve the equation $g(\eta_1,\eta_2,\eta_3) = 0$ from Lemma~\ref{lem:roots of unity ratio}.

We first note that $g$ is invariant under the group $G$ generated by the substitutions
\begin{equation} \label{eq:symmetries of equation}
(z_1, z_2, z_3) \mapsto (z_1^{-1}, z_2^{-1}, z_3^{-1}), (z_2, z_1, z_3), (z_1, z_2^{-1}, -z_1/z_3).
\end{equation}

We next apply the results of \S\ref{sec:relations}.
\begin{lemma} \label{L:roots of unity partition}
In the ring $R$, define the element $u = -z_1 z_2 z_3^{-2}$ and the subset
\[
S = \{z_1^{\pm 1}, z_2^{\pm 1}, z_3^{\pm 1}, -(z_1/z_3)^{\pm 1}, 
-(z_2/z_3)^{\pm 1}, (z_1 z_2/z_3)^{\pm 1} \}.
\]
Then for any $\eta_1, \eta_2, \eta_3 \in \mu$ with $g(\eta_1, \eta_2, \eta_3) = 0$, 
there exists an element $h \in R$ 
with $h(\eta_1, \eta_2, \eta_3) = 0$ of one of the following forms:
\begin{enumerate}
\item[(a)]
$h = \eta \pm 1$ for some $\eta \in S \cup \{u,u^{-1}\}$; or
\item[(b)]
$h = u + u^{-1} + \sum_{\eta \in T} \eta$ for some subset $T$ of $S$ with $\#T \leq 6$ which is invariant under the substition $z_i \mapsto z_i^{-1}$.
\end{enumerate}
\end{lemma}
\begin{proof}
We apply Theorem~\ref{T:relations of weight 16} to the cyclotomic relation of weight 16 obtained from $g$ by separating into monomials with coefficients $\pm 1$, taking $u$ and $u^{-1}$ twice each, then evaluating at $z_i = \eta_i$.
If any of the 16 terms is equal to its own conjugate, then (a) holds; we may thus assume that this does not occur.

By Theorem~\ref{T:relations of weight 16}, this cyclotomic relation admits a partition into indecomposable cyclotomic relations; by Lemma~\ref{lem:conjugate stable relations}, this partition can further be taken to be stable under complex conjugation.
If we group together conjugate pairs, the two evaluations of $u$ must end up in separate parts; hence one of these parts has weight at most 8, and so (b) holds.
\end{proof}

We finally proceed to solve the desired equation.
\begin{lemma} \label{lem:roots of unity solutions}
If $\eta_1, \eta_2, \eta_3 \in \mu$ satisfy $g(\eta_1, \eta_2, \eta_3) = 0$, then either there exists $\zeta \in \mu$ with
\begin{equation} \label{eq:one-parameter solutions}
\eta_1 = \eta_2 = \zeta, \eta_3 = 1 \mbox{\quad or \quad} \eta_1 = \eta_2 = \zeta, \eta_3 = -\zeta,
\end{equation}
or $\eta_1, \eta_2, \eta_3 \in \QQ(\mu_N)$ for $N \in \{15, 21, 24\}$; more precisely the orders of $\eta_1, \eta_2, \eta_3$ must fit one of the patterns listed in Table~\ref{table:sporadic orders}.
\end{lemma}
\begin{table}[ht]
\begin{tabular}{c|c|c}
Order of $\eta_1$ & Order of $\eta_2$ & Orders of $\eta_3$ \\
\hline
1 & 2 & 8 \\
1 & 4 & 24 \\
2 & 2 & 4 \\
2 & 4 & 6,12 \\
3 & 30 & 10, 15, 30 \\
4 & 4 & 3, 12 \\
6 & 7 & 21 \\
7 & 7 & 7, 14 \\
30 & 30 & 5, 6, 10, 15, 30
\end{tabular}
\caption{Orders of elements of sporadic solutions in Lemma~\ref{lem:roots of unity solutions}.}
\label{table:sporadic orders}
\end{table}
\begin{proof}
Consider the elements $h \in R$ of the forms given by Lemma~\ref{L:roots of unity partition} as the vertices of a Cayley graph for the action of the group $G$
and the generators specified above, also adding an edge from $h$ to $g-h$ when both occur as vertices. Using \SageMath{}, we compute a set $U$ of representatives of the connected components of this graph; we have $\#U = 16$. 
By Lemma~\ref{L:roots of unity partition}, our original triple $(\eta_1,\eta_2,\eta_3)$ is $G$-equivalent to one which is a zero of some $h \in U$.

For each $h \in U$, we compute the \emph{torsion closure} of the ideal $(g,h)$ of $R$ 
in the sense of \cite[\S 7]{space-vectors}, i.e., the maximal ideal of $R$
whose support contains every solution to $g(\eta_1,\eta_2,\eta_3) = h(\eta_1, \eta_2, \eta_3) = 0$ with $\eta_1, \eta_2, \eta_3 \in \mu$.
This computation in \SageMath{} (for all $h$) takes under 5 minutes on a laptop,
using the algorithm described in \cite[\S 7]{space-vectors} as implemented in \cite{space-vectors-repo}.

Of the associated primes of the resulting torsion closures, two are the one-dimensional ideals corresponding to the parametric solutions \eqref{eq:one-parameter solutions}. The rest are zero-dimensional ideals corresponding to Galois orbits of sporadic solutions in $\QQ(\mu_N)$ for $N \in \{15, 21, 24\}$; we construct Table~\ref{table:sporadic orders} by direct inspection.
\end{proof}

\begin{remark}
We sketch an alternate proof of Lemma~\ref{lem:roots of unity solutions} that can in principle be carried out by hand, although we do not do this completely here.

Suppose first that $h \in U$ arises from case (a) of Lemma~\ref{L:roots of unity partition}.
Modulo the ideal $(h,2)$ of $R$, we may write $g$ as indicated, and then deduce the indicated consequences from Remark~\ref{remark:mod 2}.
(The listed values of $h$ cover all $G$-orbits.)
\begin{center}
\begin{tabular}{c|c|c}
$h$ & $g$ modulo $(h,2)$ & Consequence \\ \hline
$z_1 \pm 1$ & $z_2 + z_2^{-1}$ & $\eta_2^4 = 1$ \\
$z_1 z_2 z_3^{-1} \pm 1$ & $z_3+ z_3^{-1}$ & $\eta_3^4 = 1$\\
$z_1 z_2 z_3^{-2} \pm 1$ & $z_1^{-2} (z_1 + z_2) (z_1 + z_2^{-1})$ &
$\eta_1^4 = \eta_2^{\pm 4}$
\end{tabular}
\end{center}
Similarly, if $h$ arises from case (b) with $T = \emptyset$, 
then $\eta_1 \eta_2 \eta_3^{-2} = \pm i$
and $g \equiv z_1^{-2} (z_1+z_2)(z_1 + z_2^{-1}) \pmod{h, 1+i}$, so by Remark~\ref{remark:mod 2} we have $\eta_1^8 = \eta_2^{\pm 8}$.
With further calculation, this yields the parametric solutions \eqref{eq:one-parameter solutions} plus solutions with $\eta_1,\eta_2,\eta_3 \in \QQ(\mu_N)$ for $N = 24$.

Suppose next that $h \in U$ arises from case (b) of Lemma~\ref{L:roots of unity partition}. We have already treated the case $\#T = 0$.
If $\#T = 2$, then $T = \{\mu, \mu^{-1}\}$ for some $\mu \in S$;
the cases $\mu = z_2, -z_2 z_3^{-1}$ cover all $G$-orbits.
Since Table~\ref{table:minimal relations} contains no entries with $n=4$,
either $u + \mu$ or $u + \mu^{-1}$ evaluates to 0.
With further calculation, this yields the parametric solutions \eqref{eq:one-parameter solutions} plus solutions with $\eta_1,\eta_2,\eta_3 \in \QQ(\mu_N)$ for $N \in \{7, 12, 15\}$.

If $\#T = 4$, then up to $G$-action we may take $T$ to be one of four options. One of these is
\[
T = \{(-z_1 z_3^{-1})^{\pm 1}, (-z_2 z_3^{-1})^{\pm 1}\},
\]
in which we have the algebraic relation $(-z_1 z_3^{-1})(-z_2 z_3^{-1}) = -u$.
In this case, we first solve the equation $h(\eta_1,\eta_2,\eta_3) = 0$ in roots of unity; this yields solutions with $\eta_1 \eta_3^{-1}, \eta_2 \eta_3^{-1} \in \QQ(\mu_N)$ with $N = 24$. We then substitute into $g$ and solve further to confirm that $\eta_1,\eta_2,\eta_3 \in \QQ(\mu_N)$.

The remaining options for $T$ with $\#T = 4$ are
\[
\{z_1^{\pm 1}, z_2^{\pm 1}\},
\{z_2^{\pm 1}, (-z_2 z_3^{-1})^{\pm 1}\},
\{z_3^{\pm 1}, (-z_2 z_3^{-1})^{\pm 1}\}.
\]
By Lemma~\ref{lem:conjugate stable relations}, either $h$ corresponds to the unique sequence listed in Table~\ref{table:minimal relations} with $n=6$ rotated by $\pm 1$,
or $h$ splits into two conjugate relations of type $R_3$.
In the first case, we must have $\eta_1, \eta_2, \eta_3 \in \QQ(\mu_N)$
for $N = 15$ without further calculation. (When $T = \{z_1^{\pm 1}, z_2^{\pm 1}\}$ solving for $\eta_3$ requires a square root, but ends up not forcing an increase in $N$.)
In the second case, $u$ equals a cube root of unity times one of
$z_2^{\pm 1}, (-z_2 z_3^{-1})^{\pm 1}$;
with further calculation, this yields solutions with $\eta_1,\eta_2,\eta_3 \in \QQ(\mu_N)$ for $N \in \{15, 24\}$.

If
$\#T = 6$, then both $h$ and $g-h$ correspond to sequences listed in 
Table~\ref{table:minimal relations} rotated by $\pm 1$. 
(Note that $h$ cannot be the sum of two disjoint conjugate minimal sequences because Table~\ref{table:minimal relations} contains no entries with $n=4$.)
In particular,
since $z_1, z_2, z_3$ each appear in either $h$ or $g-h$, we must have
$\eta_1, \eta_2, \eta_3 \in \QQ(\mu_N)$ for $N \in \{15, 21\}$ without further calculation.
\end{remark}

\section{Geometric decompositions}
\label{sec:geometric decompositions}

With Lemma~\ref{lem:roots of unity solutions} in hand,
we now proceed to the proof of Theorem~\ref{T:main} and Corollary~\ref{C:main}.

\begin{proof}[Proof of Theorem~\ref{T:main}]
As noted earlier, the LMFDB asserts this result in all cases where $\dim(A) \leq 6$, so we may ignore those cases in what follows.
Since the cases $n = 1,2,7,30$ are excluded by our dimension bound, 
Theorem~\ref{T:madan-pal} implies that the real Weil polynomial associated to $A$ equals $P_n(3-x)$. 

Suppose that $\alpha_1, \alpha_2$ are distinct Frobenius eigenvalues of $A$ such that
$\eta_3 = \alpha_1/\alpha_2$ is a root of unity. By Lemma~\ref{lem:roots of unity ratio}, the roots of unity $\eta_1, \eta_2$ of order $n$ corresponding to $\alpha_1, \alpha_2$
via \eqref{eq:MP eigenvalue} satisfy $g(\eta_1, \eta_2, \eta_3) = 0$.
Note that each of the exceptional cases listed in Table~\ref{table:sporadic orders} includes a now-excluded value of $n$, and that the first family listed in \eqref{eq:one-parameter solutions} is inconsistent with the hypothesis that $\alpha_1 \neq \alpha_2$.
Consequently,  Lemma~\ref{lem:roots of unity solutions} implies that $\eta_1, \eta_2, \eta_3$ must belong to the second family listed in \eqref{eq:one-parameter solutions}; that is, $\alpha_1$ and $\alpha_2$ are the two roots of \eqref{eq:MP eigenvalue} for some common value of $\eta$. 

From the previous calculation, it follows that for any positive integer $m$, the Weil polynomial associated to $A_{\FF_{2^m}}$ is either irreducible or the square of an irreducible Weil polynomial. Moreover, the latter outcome occurs when $m=n$:
for each root of unity $\eta$ of order $n$, both roots of \eqref{eq:MP eigenvalue} occur as Frobenius eigenvalues of $A_n$ and their ratio is $-\eta$.
We deduce that the Weil polynomial associated to $A_{\FF_{q^n}}$ is the square of an irreducible Weil polynomial $Q(x)$, and moreover no two of the roots of $Q(x)$ have ratio equal to a nontrivial root of unity.

If $n$ is not a power of 2, then $A$ is ordinary by Lemma~\ref{lem:Newton polygon}, as then is $A_{\FF_{q^n}}$. Hence Corollary~\ref{cor:Honda-Tate ordinary} implies that $Q(x)$ is the Weil polynomial associated to an abelian variety over $\FF_{q^n}$,
so we may apply Lemma~\ref{lem:geometrically simple criterion} to deduce that $A_{\overline{\FF}_q}$ is isogenous to the square of an abelian variety.

If $n$ is a power of 2, then we may use our dimension bound to reduce to the case $n \geq 8$. 
Lemma~\ref{lem:Newton polygon} implies that the slopes of the Newton polygon of $A$ are all equal to $2/n$ or $1-2/n$;
we may thus apply Lemma~\ref{lem:geometrically simple by Newton polygon} to deduce that $A$ is geometrically simple.
\end{proof}

\begin{proof}[Proof of Corollary~\ref{C:main}]
By Theorem~\ref{T:madan-pal}, there exists a unique integer $n_i$ such that $B_i$ is an isogeny factor of $A_{n_i}$.
Let $\alpha_i$ be a Frobenius eigenvalue of $B_i$.
By Theorem~\ref{T:Honda-Tate}, 
$\Hom(B_{1,\overline{\FF}_q}, B_{2,\overline{\FF}_q}) \neq 0$
if and only if some conjugate of $\alpha_2$ equals $\alpha_1$ times a root of unity. Using this criterion, it is straightforward to verify the ``if'' implication.

To check the ``only if'' implication, we may assume that
$B_1$ and $B_2$ are not isogenous and that $\alpha_1/\alpha_2$ is a root of unity.
If $n_1 = n_2$, then by Theorem~\ref{T:madan-pal} the common value is in $\{7, 30\}$. If $n_1 \neq n_2$, we may choose a root of unity $\eta_i$ of order $n_i$ associated to $\alpha_i$ via Lemma~\ref{lem:MP eigenvalue}; by Lemma~\ref{lem:roots of unity solutions}, $n_1$ and $n_2$ must appear as the first two entries in some row of Table~\ref{table:sporadic orders}.
\end{proof}

\section{Additional properties}
\label{sec:additional properties}

In \cite{vanbommel-etc}, it is shown for any fixed $q$, any sufficiently large positive integer is the order of some abelian variety over $\FF_q$ which is simultaneously ordinary, geometrically simple, and principally polarizable. It is thus natural to ask how these conditions interact for simple abelian varieties of order 1 over $\FF_2$.

We first consider the combination of the ordinary and geometrically simple conditions. From the proof of Theorem~\ref{T:main} and our calculation of Newton polygons (Lemma~\ref{lem:Newton polygon}), we obtain the following.
\begin{corollary}
Let $A$ be a simple abelian variety over $\FF_2$ of order $1$.
Then $A$ cannot be both ordinary and geometrically simple.
In addition, if $\dim(A) \geq 3$, then $A$ is either ordinary or geometrically simple (but not both).
\end{corollary}

As for the principally polarizable condition, we have the following partial result.
\begin{theorem}
Let $n$ be a power of an odd prime $p$. Then $A_n$ is isogenous to a principally polarizable abelian variety.
\end{theorem}
\begin{proof}
Let $\alpha$ be a Frobenius eigenvalue of $A_n$
and choose $\eta \in \mu$ of order $n$ for which \eqref{eq:MP eigenvalue}  holds;
then $\eta \in \QQ(\alpha)$. Define
\[
\beta = \alpha + 2\alpha^{-1}, \qquad \xi = \eta + \eta^{-1};
\]
we derive from \eqref{eq:MP eigenvalue} the equation
\[
\beta^2 + (\xi-2) \beta - (2+3\xi) = 0,
\]
which as a quadratic equation in $\beta$ has discriminant
\[
(\xi-2)^2 + 4(2+3\xi) = (\xi + 2)(\xi + 6).
\]
We now have the following field diagram:
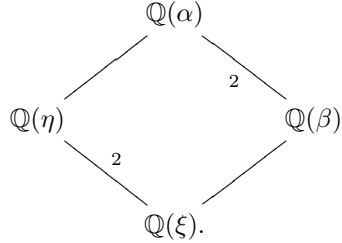
\[
\xymatrix{
& \QQ(\alpha) \ar@{-}[ld] \ar@{-}_2[rd] \\
\QQ(\eta) \ar@{-}^2[rd] & & \QQ(\beta) \ar@{-}[ld] \\
& \QQ(\xi).
}
\]
Let $v$ be a prime of $\QQ(\xi)$ lying above $p$; then modulo $v$, $\xi$ is congruent to 2 and so $(\xi+2)(\xi+6)$ is congruent to 32. As a result, $v$ does not ramify in $\QQ(\beta)$;
since $v$ does ramify in $\QQ(\eta)$, every prime of $\QQ(\beta)$ above $v$ ramifies in $\QQ(\alpha)$.
We may thus apply \cite[Theorem~1.1]{howe} to deduce 
that $A_n$ is isogenous to a principally polarizable abelian variety.
\end{proof}

\begin{remark} \label{rmk:LMFDB principally polarizable}
For positive integers $n$ such that the simple isogeny factors of $A_n$ are of dimension at most 6, LMFDB reports the following as to whether or not these factors are isogenous to a principally polarizable abelian variety.
\begin{itemize}
\item Yes: $n=1, 2, 3, 4, 5, 7, 9$.
\item No: $n=6, 10, 12, 14, 18, 30$.
\item Unknown: $n=8$.
\end{itemize}
It may be possible to collect more data using the methods of \cite{marseglia} and \cite{bergstrom-karemaker-marseglia}, especially when $A_n$ is ordinary (which by Lemma~\ref{lem:Newton polygon} occurs when $n$ is not a power of 2).
\end{remark}

\end{document}